\documentclass[10pt,a4paper,reqno]{amsart}
\usepackage{epsfig}
\usepackage{color}
\usepackage{amssymb,amsmath,amsthm,amstext,amsfonts}
\usepackage{amssymb,amsmath,amsthm,amstext,amsfonts}
\usepackage{amsmath,amstext,amsthm,amsfonts}
\usepackage{color}
\usepackage[mathscr]{eucal}

\usepackage{hyperref}
\RequirePackage[dvipsnames]{xcolor} 
\definecolor{halfgray}{gray}{0.55} 
\definecolor{webgreen}{rgb}{0,0.5,0}
\definecolor{webbrown}{rgb}{.6,0,0} \hypersetup{%
  colorlinks=true, linktocpage=true, pdfstartpage=3,
  pdfstartview=FitV,%
  breaklinks=true, pdfpagemode=UseNone, pageanchor=true,
  pdfpagemode=UseOutlines,%
  plainpages=false, bookmarksnumbered, bookmarksopen=true,
  bookmarksopenlevel=1,%
  hypertexnames=true,
  pdfhighlight=/O,
  urlcolor=webbrown, linkcolor=RoyalBlue,
  citecolor=webgreen, 
  pdftitle={},%
  pdfauthor={},%
  pdfsubject={2000 MAthematical Subject Classification: Primary:},%
  pdfkeywords={},%
  pdfcreator={pdfLaTeX},%
  pdfproducer={LaTeX with hyperref}%
}

\usepackage{psfrag}
\usepackage{url}
\usepackage{epstopdf}
\usepackage{epic,eepic}
\usepackage{upgreek}
\usepackage{amssymb}
\usepackage{mathrsfs}
\usepackage{verbatim}
\usepackage{mathrsfs}
\usepackage{amstext}
\usepackage{amsthm}
\usepackage{amssymb}
\usepackage{graphicx}

\theoremstyle{plain}
\newtheorem{theorem}{Theorem}[section]
\newtheorem{proposition}[theorem]{Proposition}
\newtheorem{lemma}[theorem]{Lemma}
\newtheorem{corollary}[theorem]{Corollary}
\newtheorem{maintheorem}{Theorem}

\theoremstyle{definition}
\newtheorem{remark}[theorem]{Remark}
\newtheorem{example}[theorem]{Example}

\begin{document}
\title[A variational principle for the metric mean dimension of level sets]{A variational principle for the metric mean dimension of level sets}

\author{Lucas Backes}
\author{Fagner B. Rodrigues}

\address{\noindent Departamento de Matem\'atica, Universidade Federal do Rio Grande do Sul, Av. Bento Gon\c{c}alves 9500, CEP 91509-900, Porto Alegre, RS, Brazil.}

\email{lucas.backes@ufrgs.br}
\email{fagnerbernardini@gmail.com}

\date{\today}




\begin{abstract}
We prove a variational principle for the upper and lower metric mean dimension of level sets
\begin{displaymath}
\left\{x\in X: \lim_{n\to\infty}\frac{1}{n}\sum_{j=0}^{n-1}\varphi(f^{j}(x))=\alpha\right\}
\end{displaymath}
associated to continuous potentials $\varphi:X\to \mathbb R$ and continuous dynamics $f:X\to X$ defined on compact metric spaces and exhibiting the specification property. This result relates the upper and lower metric mean dimension of the above mentioned sets with growth rates of measure-theoretic entropy of partitions decreasing in diameter associated to some special measures. Moreover, we present several examples to which our result may be applied to. Similar results were previously known for the topological entropy and for the topological pressure.
\end{abstract}

\maketitle

\section{Introduction}

One of the most important notions in Dynamical Systems is that of \emph{topological entropy}. It is a topological invariant and, roughly speaking, measures how chaotic a system is. In particular, it is an effective tool to decide whether two systems are conjugated or not. Nevertheless, there are plenty of systems with infinite topological entropy (for instance, they form a $C^0$-generic set in the space of homeomorphisms of a compact manifold \cite{Yano} with dimension greater than one) and thus, in this context, the entropy is not useful anymore. Therefore, in order to study these types of systems, new dynamical quantities are required and an example of such a quantity is the \emph{metric mean dimension}. 

The notion of metric mean dimension was introduced by Lindenstrauss and Weiss in \cite{LW} as metric-dependent analog of the \emph{mean dimension}, a topological invariant associated to a dynamical system which was introduced by Gromov \cite{Gromov}. This last notion has several applications, like in the study of embedding problems \cite{GT}, and the metric mean dimension presents an upper bound to it. But more than that, the metric mean dimension turned out to be useful in several contexts like in the study of compression \cite{GS2, GS3}.

In the present paper we give a modest contribution to the study of ergodic theoretical aspects of the metric mean dimension by presenting a variational principle. Previous connections between ergodic theory and metric mean dimension were presented, for instance, by Lindenstrauss and Tsukamoto \cite{LT}, Velozo and Velozo \cite{VV}, Tsukamoto \cite{TSU2020},  Shi \cite{Shi}, Gutman and \'Spiewak \cite{GS} and Yang, Chen and Zhou \cite{YCZ}. For more on these works, see Section \ref{sec: related results}. The main novelty of our work with respect to the previously mentioned ones is that our variational principle holds for special subsets and not only for the whole phase space. More precisely, we consider level sets 
\begin{displaymath}
K_\alpha=\left\{x\in X: \lim_{n\to\infty}\frac{1}{n}\sum_{j=0}^{n-1}\varphi(f^{j}(x))=\alpha\right\}
\end{displaymath}
associated to continuous potentials $\varphi:X\to \mathbb R$ and continuous dynamics $f:X\to X$ defined on compact metric spaces exhibiting the specification property and present a relation between the \emph{upper and lower metric mean dimension} of the above mentioned sets and growth rates of measure-theoretic entropy of partitions decreasing in diameter associated to some special measures. This is the content of Theorem \ref{thm1}. In Section \ref{sec: examples} we present several examples to which our result is applicable. 

\subsection{Mutlifractal analysis} 
The general idea of multifractal analysis consists in decomposing the phase space into subsets of points with similar dynamical behavior, for instance, in sets of points with the same Birkhoff average, the same Lyapunov exponents or the same local entropies, and to describe the size of each of such subsets from a geometrical or topological viewpoint. The information (collection of numbers) obtained via this procedure for one such decomposition of the phase space is called a \emph{multifractal spectrum}. Then, in the best-case scenario the idea is that if one knows some of these spectra one could fully recover the dynamics (see for instance \cite{BPS,BPS2}). This phenomenon is sometimes called \emph{multifractal rigidity}. But even when we are not in such a nice world, we still can get useful information about the dynamics from these various spectra (see for instance \cite{Cli, Olsen, STVW, TV2}).
Our main result, Theorem \ref{thm1}, may be seen as a small contribution to the study of one such spectra, namely, the one obtained by measuring the size of level sets of Birkhoff averages with respect to the metric mean dimension. In particular, as a consequence of our result we get that the map $\alpha \mapsto \mathrm{\underline{mdim}_M}\,\Big(K_\alpha,f, d\Big)$ is concave when restricted to the set of parameters $\alpha\in \mathbb R$ for which $K_\alpha\neq \emptyset \empty$. As far as we know, this is the first time this spectrum was considered and we hope that our results may be of some more help in the study of multifractal analysis of systems with infinite topological entropy.

\section{Definitions and Statements}\label{sec: def and state}
Let $(X,d)$ be a compact metric space and $f \colon X \to X$ be a continuous map. Given $n\in \mathbb{N}$, we define the dynamical metric $d_n \colon X \times X \, \to \,[0,\infty)$ by
\begin{displaymath}
d_n(x,z)=\max\,\Big\{d(x,z),\,d(f(x),f(z)),\,\dots,\,d(f^{n-1}(x),f^{n-1}(z))\Big\}.
\end{displaymath}
It is easy to see that $d_n$ is indeed a metric and, moreover, generates the same topology as $d$. Furthermore, given $\varepsilon > 0$, $n \in \mathbb{N}$ and a point $x \in X$, we define the open $(n, \varepsilon)$-ball around $x$ by
\begin{displaymath}
\mathit{B}_{n}(x, \varepsilon) = \{ y \in X ; d_n( x, y) < \varepsilon \}.
\end{displaymath}
We sometimes call these $(n, \varepsilon)$-balls \emph{dynamical balls} of radius $\varepsilon$ and length $n$. We say that a set $E \subset X$ is \emph{$(n,\varepsilon)$--separated} by $f$ if $d_n(x,z) > \varepsilon$ for every $x,z \in E$.

\subsection{The metric mean dimension}\label{sse.mmd}

Given $n\in \mathbb{N}$ and $\varepsilon>0$, let us denote by $s(f,n,\varepsilon)$ the maximal cardinality of all $(n,\varepsilon)$--separated subsets of $X$ by $f$ which, due to the compactness of $X$, is finite. 

The \emph{upper metric mean dimension} of $f$ with respect to $d$ 
is given by 
$$\mathrm{\overline{mdim}_M}\,\Big(X,f,d\Big) = \limsup_{\varepsilon\,\to\, 0} \,\frac{h(f,\varepsilon)}{|\log \varepsilon|}$$
where
$$h(f,\varepsilon) = \limsup_{n\, \to\, \infty}\,\frac{1}{n}\,\log s(f,n,\varepsilon).$$
Similarly, the \emph{lower metric mean dimension} of $f$ with respect to $d$ is given by 
$$\mathrm{\underline{mdim}_M}\,\Big(X,f,d\Big) = \liminf_{\varepsilon\,\to\, 0} \,\frac{h(f,\varepsilon)}{|\log \varepsilon|}.$$
In the case when $\mathrm{\underline{mdim}_M}\,\Big(X,f,d\Big)=\mathrm{\overline{mdim}_M}\,\Big(X,f,d\Big)$ this common value is called the \emph{metric mean dimension} of $f$ with respect to $d$ and is denoted simply by $\mathrm{mdim_M}\,\Big(X,f,d\Big)$.

Recall that the \emph{topological entropy} of the map $f$ is given by
\begin{displaymath}
h_{\text{top}}(f) = \lim_{\varepsilon \, \to \, 0}\,h(f,\varepsilon).
\end{displaymath}
Consequently, $ \mathrm{\overline{mdim}_M}\,\Big(X,f,d\Big) =\mathrm{\underline{mdim}_M}\,\Big(X,f,d\Big)= 0$ whenever the topological entropy of $f$ is finite. In particular, the metric mean dimension is a suitable quantity to study systems with infinite topological entropy. For more on these quantities see \cite{LW,LT,TSU2020} and references therein.  

\subsection{The metric mean dimension for non-compact subset}  We now present the notion of metric mean dimension on non-compact sets introduced in \cite{CHeng}.
Given a set $Z\subset X$, let us consider
$$
m(Z,s, N,\varepsilon)=\inf_{\Gamma}\left\{\sum_{i\in I}\exp{\left(-sn_i\right)}\right\},
$$
where the infimum is taken over all covers $\Gamma=\{B_{n_i}(x_i,\varepsilon)\}_{i\in I}$ of $Z$ with $n_i\geq N$. We also consider
$$
m(Z,s,\varepsilon)=\lim_{N\to\infty}m(Z,s, N,\varepsilon).
$$
One can show (see for instance \cite{Pesin}) that there exists a certain number $s_0\in [0,+\infty)$ such that $m(Z,s,\varepsilon)=0$ for every $s>s_0$ and $m(Z,s,\varepsilon)=+\infty$ for every $s<s_0$. In particular, we may consider
$$
h\,\Big(Z,f,\varepsilon\Big)=\inf\{s:m(Z,s,\varepsilon)=0\}=\sup\{s:m(Z,s,\varepsilon)=+\infty\}.
$$
The \textit{upper metric mean dimension of f on $Z$} is then defined as the following limit
\begin{equation*}
    \overline{\text{mdim}}_M\,\Big(Z,f,d\Big)=\limsup_{\varepsilon\to 0}\frac{h\,\Big(Z,f,\varepsilon\Big)}{|\log \varepsilon|}.
\end{equation*}
Similarly, the \textit{lower metric mean dimension of f on $Z$} is defined as
\begin{equation*}
    \underline{\text{mdim}}_M\,\Big(Z,f,d\Big)=\liminf_{\varepsilon\to 0}\frac{h\,\Big(Z,f,\varepsilon\Big)}{|\log \varepsilon|}.
\end{equation*}
In the case when $Z=X$ one can check that the two definitions of upper/lower metric mean dimension given above actually coincide.

\subsection{Level sets of a continuous map} \label{sec: level sets}
Let $C(X,\mathbb{R})$ denote the set of all continuous maps $\varphi:X\to\mathbb R$ and take $\varphi \in C(X,\mathbb{R})$. For $\alpha\in\mathbb R$, let
 \begin{align}\label{def-target-set}
     K_\alpha=\left\{x\in X: \lim_{n\to\infty}\frac{1}{n}\sum_{j=0}^{n-1}\varphi(f^{j}(x))=\alpha\right\}.
 \end{align}
We also consider the set 
\begin{displaymath}
\mathcal L_\varphi=\{\alpha\in\mathbb R: K_\alpha\not=\emptyset\}.
\end{displaymath}
It is easy to see that $\mathcal L_\varphi $ is a bounded and non-empty set \cite[Lemma 2.1]{TV}. Moreover, if $f$ satisfies the so called specification property (see Section \ref{sec: specification}) then $\mathcal L_\varphi$ is an interval of $\mathbb R$ and, moreover, $\mathcal L_\varphi=\{\int \varphi d\mu ; \mu \in \mathcal{M}_f(X)\}$ where $\mathcal{M}_f(X)$ stands for the set of all invariant measures (see \cite[Lemma 2.5]{Tho}).

\subsection{The auxiliary quantities $\Lambda_\varphi\mathrm{\overline{mdim}_M}\,(f,\alpha,d)$ and $\Lambda_\varphi\mathrm{\underline{mdim}_M}\,(f,\alpha,d)$}
Fix  $\alpha\in\mathbb R$ and $\varphi\in C(X,\mathbb{R})$. For $\delta>0$ and $n\in\mathbb N$ define the set
\begin{align*}
    P(\alpha,\delta,n)=\left\{x\in X: \left|\frac{1}{n}\sum_{j=0}^{n-1}\varphi(f^{j}(x))-\alpha\right|<\delta\right\}.
\end{align*}
Let $N(\alpha,\delta,n,\varepsilon)$ denote the minimal number of $(n,\varepsilon)$-balls needed to cover $P(\alpha,\delta,n)$. Define 
\begin{align*}
    \Lambda_{\varphi}(\alpha,\varepsilon)=\lim_{\delta\to0}\liminf_{n\to\infty}\frac{1}{n}\log  N(\alpha,\delta,n,\varepsilon)
\end{align*}
and 

\begin{align}\label{def_Lamb}
\begin{split}
    \Lambda_\varphi\mathrm{\overline{mdim}_M}\,(f,\alpha,d)&=\limsup_{\varepsilon\to0}\frac{\Lambda_{\varphi}(\alpha,\varepsilon)}{|\log\varepsilon|},\\
    \Lambda_\varphi\mathrm{\underline{mdim}_M}\,(f,\alpha,d)&=\liminf_{\varepsilon\to0}\frac{\Lambda_{\varphi}(\alpha,\varepsilon)}{|\log\varepsilon|}.
\end{split}    
\end{align}

\begin{remark} \label{remark: sep x span}
Observe that, if $M(\alpha,\delta,n,\varepsilon)$ denotes the maximal cardinality of a $(n,\varepsilon)$-separated set contained in $P(\alpha,\delta,n)$, then we have that 
$$
N(\alpha,\delta,n,\varepsilon)\leq M(\alpha,\delta,n,\varepsilon)\leq N(\alpha,\delta,n,\varepsilon/2).
$$
In particular,
\begin{align}\label{eq:ineq-sep-cov}
    \Lambda_{\varphi}(\alpha,\varepsilon)=\lim_{\delta\to0}\liminf_{n\to\infty}\frac{1}{n}\log  M(\alpha,\delta,n,\varepsilon).
\end{align}
\end{remark}

\subsection{The main quantities $\mathrm{H_\varphi\overline{mdim}_M}(f,\alpha,d)$ and $\mathrm{H_\varphi\underline{mdim}_M}(f,\alpha,d)$} \label{sec: def H varphi}

Given $\varphi\in C(X,\mathbb R)$ and $\alpha \in\mathbb R$, let us consider
 $$\mathcal M_{f}(X,\varphi,\alpha)=\left\{\mu\text{ is $f$-invariant and }\int \varphi\;d\mu=\alpha\right\}.$$
A simple observation is that $\mathcal M_{f}(X,\varphi,\alpha)\neq \emptyset$ for every $\alpha\in \mathcal{L}_\varphi$ (see \cite[Lemma 4.1]{TV}).

Let $\mu\in\mathcal M_f(X)$. We say that $\xi=\{C_1,\ldots,C_k\}$ is a measurable partition of $X$ if every $C_i$ is a measurable set, $\mu\left(X\setminus\cup_{i=1}^kC_i\right)=0$ and $\mu\left(C_i\cap C_j\right)=0$ for every $i\neq j$. The \emph{entropy} of $\xi$ with respect to $\mu$ is given by
\[
H_\mu(\xi)=-\sum_{i=1}^{k}\mu(C_i)\log(\mu(C_i)).
\]
Given a measurable partition $\xi$, we consider $\xi^n=\bigvee_{j=0}^{n-1}f^{-j}\mathcal \xi$. Then, the \emph{metric entropy of $(f,\mu)$ with respect to $\xi$} is given by
\[
h_\mu(f,\xi)=\lim_{n\to +\infty}\frac{1}{n} H_\mu(\xi^n).
\] 
Using this quantity we define
\begin{align}\label{def_Hmdim} 
   \mathrm{H_\varphi\overline{mdim}_M}\,(f,\alpha,d)=\limsup_{\varepsilon\to0}\frac{1}{|\log\varepsilon|}\sup_{\mu\in \mathcal M_{f}(X,\varphi,\alpha)}\inf_{|\xi|<\varepsilon}h_\mu(f,\xi)
\end{align}
and
\begin{align*} 
   \mathrm{H_\varphi\underline{mdim}_M}\,(f,\alpha,d)=\liminf_{\varepsilon\to0}\frac{1}{|\log\varepsilon|}\sup_{\mu\in \mathcal M_{f}(X,\varphi,\alpha)}\inf_{|\xi|<\varepsilon}h_\mu(f,\xi)
\end{align*}
where $|\xi|$ denotes the diameter of the partition $\xi$ and the infimum is taken over all finite measurable partitions of $X$ satisfying $|\xi|<\varepsilon$. 

We also recall that the \emph{metric entropy of $(f,\mu)$} is given by
\[
h_\mu(f)=\sup_\xi h_\mu(f,\xi)
\]
where the supremum is taken over all finite measurable partitions $\xi$ of $X$.

\subsection{Specification property} \label{sec: specification}
We say that $f$ satisfies the \emph{specification property} if for every $\epsilon > 0$, there exists an integer $m = m(\epsilon )$ such that for any collection of finite intervals $ I_j = [a_j, b_j ] \subset \mathbb{N}$, $j = 1, \ldots, k $, satisfying $a_{j+1} - b_j \geq m(\epsilon )$ for every $j = 1, \ldots, k-1 $ and any $x_1, \ldots, x_k$ in $X$, there exists a point $x \in X$ such that
\begin{equation*}
d(f^{p + a_j}x, f^p x_j) < \epsilon \mbox{ for all } p = 0, \ldots, b_j - a_j \mbox{ and every } j = 1, \ldots, k.
\end{equation*}
The specification property is present in many interesting examples. For instance, every topologically mixing locally maximal hyperbolic set has the specification property and factors of systems with specification have specification (see for instance \cite{KH}). Other examples of systems satisfying this property which are more adapted to our purposes will appear in Section \ref{sec: examples}.

\subsection{Main result} Our main result may be seen as an extension of \cite[Theorem 5.1]{TV} to the infinite entropy setting. 
\begin{maintheorem}\label{thm1}
Suppose $f:X \to X$ is a continuous transformation with the specification
property. Let $\varphi\in C(X,\mathbb R)$ and $\alpha\in\mathbb R$ be such that $K_\alpha\not=\emptyset$.
Then
$$
\mathrm{\overline{mdim}_M}\,\Big(K_\alpha,f, d\Big)=\Lambda_\varphi\mathrm{\overline{mdim}_M}\,(f,\alpha,d)= \mathrm H_\varphi\mathrm{\overline{mdim}_M}\,(f,\alpha,d).
$$
Similarly,
$$
\mathrm{\underline{mdim}_M}\,\Big(K_\alpha,f, d\Big)=\Lambda_\varphi\mathrm{\underline{mdim}_M}\,(f,\alpha,d)= \mathrm H_\varphi\mathrm{\underline{mdim}_M}\,(f,\alpha,d).
$$
\end{maintheorem}
We consider the equalities between $\mathrm{\overline{mdim}_M}\,\Big(K_\alpha,f, d\Big)$ and $\mathrm H_\varphi\mathrm{\overline{mdim}_M}\,(f,\alpha,d)$ and between $\mathrm{\underline{mdim}_M}\,\Big(K_\alpha,f, d\Big)$ and $\mathrm H_\varphi\mathrm{\underline{mdim}_M}\,(f,\alpha,d)$ to be the most important part of our result because it relates a topological quantity with one that has an ergodic-theoretical flavor. Moreover, in some cases it allow us to obtain some interesting properties about the multifractal spectrum. For instance, will show bellow that

\begin{proposition}\label{prop: H is concave}
Under the assumptions of Theorem \ref{thm1}, the map 
\[\mathcal{L}_\varphi\ni \alpha\mapsto \mathrm {H_\varphi\underline{mdim}_M}\,(f,\alpha,d)\]
is concave.
\end{proposition}

Consequently, combining this result with Theorem \ref{thm1} we get that

\begin{corollary}
Under the assumptions of Theorem \ref{thm1}, the map 
\[\mathcal{L}_\varphi\ni \alpha \mapsto \mathrm{\underline{mdim}_M}\,\Big(K_\alpha,f, d\Big)\]
is concave.
\end{corollary}

An interesting question is whether we can change the order between the limit and the supremum in the definition of $\mathrm H_\varphi\mathrm{\overline{mdim}_M}\,(f,\alpha,d)$ and $\mathrm H_\varphi\mathrm{\underline{mdim}_M}\,(f,\alpha,d)$. This would allow, for instance, to talk about the existence of ``maximizing measures": measures that realize the supremum. Such a measure would capture the complexity of the system over all scales $\varepsilon >0$.  
It was observed in \cite[Section VIII]{LT} that a similar question involving different ergodic quantities is, in general, false. Nevertheless, under the additional assumption that $f$ has the marker property, one can do such a change (in the setting of \cite{LT}) as observed by Yang, Chen and Zhou \cite{YCZ}. As for our hypothesis that $f$ satisfies the specification property, we do not know whether it is actually required for Theorem \ref{thm1} to hold or if it is just an artefact of the technique. 

\subsection{Related results}\label{sec: related results}

As already mentioned, for the topological entropy a result similar to Theorem \ref{thm1} was obtained in \cite{TV}. In fact, our result was inspired by that one. Moreover, \cite{TV} was extended to the framework of topological pressure in \cite{Tho}. 

As for variational results involving the upper metric mean dimension, there are several works dealing with this problem. For instance, \cite{LT} presented a variational principle relating the metric mean dimension with the supremum of certain rate distortion functions over invariant measures of the system. This was further explored in \cite{VV}. More recently, \cite{Shi} obtained variational principles for the metric mean dimension in terms of Brin-Katok local entropy and Shapira's
entropy of an open cover. One result that is more connected to ours is the one obtained in \cite{GS} which says that
\begin{align}\label{vari_Hmdim} 
 \mathrm{\overline{mdim}_M}\,(X,f,d)=\limsup_{\varepsilon\to0}\frac{1}{|\log\varepsilon|}\sup_{\mu\in \mathcal M_{f}(X)}\inf_{|\xi|<\varepsilon}h_\mu(f,\xi)
\end{align}
and
\begin{align} \label{vari_Hmdim2}
 \mathrm{\underline{mdim}_M}\,(X,f,d)=\liminf_{\varepsilon\to0}\frac{1}{|\log\varepsilon|}\sup_{\mu\in \mathcal M_{f}(X)}\inf_{|\xi|<\varepsilon}h_\mu(f,\xi).
\end{align}
These are variational results for the upper/lower metric mean dimension of the entire space $X$ while Theorem \ref{thm1} applies also to level sets of continuous maps $\varphi$. Observe that in the case when $\varphi$ is a constant map equal to $\alpha$, the $\alpha$-level set of it coincides with $X$. In particular, whenever $f$ has the specification property, \eqref{vari_Hmdim} and \eqref{vari_Hmdim2} may be seen as particular cases of our result. We stress however that the results in \cite{GS} do not assume such property.


\section{Proofs of Theorem \ref{thm1} and Proposition \ref{prop: H is concave}} \label{sec: proofs}
In this section we present the proofs of Theorem \ref{thm1} and Proposition \ref{prop: H is concave} starting with the latter one which is much simpler.

\begin{proof}[Proof of Proposition \ref{prop: H is concave}]
Given measures $\mu_1,\mu_2\in \mathcal{M}_f(X)$, using that the map $\mu\to H_\mu(\xi)$ is concave for any finite and measurable partition $\xi$ (see \cite[Lemma 9.5.1]{BS}), it follows that for any $t\in[0,1]$,
\[ th_{\mu_1}(f,\xi)+(1-t)h_{\mu_2}(f,\xi)\leq h_{t\mu_1+(1-t)\mu_2}(f,\xi). \]
In particular,
\begin{equation}\label{eq: auxil1 concave}
t\inf_{|\xi|<\varepsilon} h_{\mu_1}(f,\xi)+(1-t)\inf_{|\xi|<\varepsilon}h_{\mu_2}(f,\xi)\leq \inf_{|\xi|<\varepsilon}h_{t\mu_1+(1-t)\mu_2}(f,\xi).
\end{equation}

Now, given $\alpha_1,\alpha_2 \in \mathcal{L}_\varphi$, by the comments in Section \ref{sec: level sets} there exist invariant measures $\mu_1,\mu_2\in \mathcal{M}_f(X)$ such $\alpha_i=\int \varphi d\mu_i$, $i=1,2$. For any $t\in[0,1]$, consider $\mu=t\mu_1+(1-t)\mu_2$ and $\alpha=t\alpha_1+(1-t)\alpha_2$. Then, $\mu\in \mathcal M_{f}(X,\varphi,\alpha)$. Combining this observation with \eqref{eq: auxil1 concave} we get that
\[
\begin{split}
&t  \sup_{\mu_1\in \mathcal M_{f}(X,\varphi,\alpha_1)} \inf_{|\xi|<\varepsilon} h_{\mu_1}(f,\xi)+(1-t)\sup_{\mu_2\in \mathcal M_{f}(X,\varphi,\alpha_2)}\inf_{|\xi|<\varepsilon}h_{\mu_2}(f,\xi)\\
&\leq \sup_{\mu_1\in \mathcal M_{f}(X,\varphi,\alpha_1),\mu_2\in \mathcal M_{f}(X,\varphi,\alpha_2)}\inf_{|\xi|<\varepsilon}h_{t\mu_1+(1-t)\mu_2}(f,\xi)\\
&\leq \sup_{\mu\in \mathcal M_{f}(X,\varphi,\alpha)}\inf_{|\xi|<\varepsilon}h_{\mu}(f,\xi).
\end{split}
\]
Then, dividing everything by $|\log \varepsilon|$, taking ``$\liminf_{\varepsilon\to 0}$'' and using that $\liminf(a) + \liminf(b)\leq \liminf (a+b)$ we get that
\[ t \mathrm H_\varphi\mathrm{\underline{mdim}_M}\,(f,\alpha_1,d)+(1-t) \mathrm H_\varphi\mathrm{\underline{mdim}_M}\,(f,\alpha_2,d)\leq  \mathrm H_\varphi\mathrm{\underline{mdim}_M}\,(f,t\alpha_1+(1-t)\alpha_2,d).\]
Consequently, the map $\mathcal{L}_\alpha\ni \alpha\mapsto  \mathrm H_\varphi\mathrm{\underline{mdim}_M}\,(f,\alpha,d)$ is concave concluding the proof of the proposition.
\end{proof}

\begin{remark}
It is not clear to us whether a version of Proposition \ref{prop: H is concave} holds for the map $\mathcal{L}_\alpha\ni \alpha\mapsto  \mathrm H_\varphi\mathrm{\overline{mdim}_M}\,(f,\alpha,d)$. In fact, in order to get the desired conclusion in the aforementioned proposition we have used the property that
   $\liminf(a) + \liminf(b)\leq \liminf (a+b)$ which obviously does not hold for the $\limsup$.
\end{remark}

We now present the the poof of Theorem \ref{thm1}. We start considering the first claim of the theorem and, for the sake of clarity of the presentation, we split it into three main propositions. We emphasize that this proof is an adaptation of the proof of Theorem 5.1 of \cite{TV} to our setting. Fix $\varphi\in C(X,\mathbb R)$ and $\alpha\in\mathbb R$ such that $K_\alpha\not=\emptyset$. Moreover, assume initially that all the quantities $\mathrm{\overline{mdim}_M}\,\Big(K_\alpha,f, d\Big)$, $\Lambda_\varphi\mathrm{\overline{mdim}_M}\,(f,\alpha,d)$ and $\mathrm H_\varphi\mathrm{\overline{mdim}_M}\,(f,\alpha,d)$ are finite.

\begin{proposition}\label{lemma:20}
Under the hypotheses of Theorem \ref{thm1} we have that
\[
\mathrm{\overline{mdim}_M}\,(K_\alpha,f,d)\leq \Lambda_\varphi\mathrm{\overline{mdim}_M}\,(f,\alpha,d).
\]
\end{proposition}
\begin{proof}

Let  $\{\varepsilon_j\}_{j\in\mathbb N}$ be a sequence of positive numbers converging to zero such that
\[
\mathrm{\overline{mdim}_M}\,(K_\alpha,f,d)=\lim_{j\to\infty}\frac{h(K_\alpha,f,\varepsilon_j)}{|\log \varepsilon_j|}.
\]
In particular we have that 
$$
\limsup_{j\to\infty}\frac{\Lambda_{\varphi}(\alpha,\varepsilon_j)}{|\log\varepsilon_j|}\leq\limsup_{\varepsilon\to0}\frac{\Lambda_{\varphi}(\alpha,\varepsilon)}{|\log\varepsilon|}= \Lambda_\varphi\mathrm{\overline{mdim}_M}\,(f,\alpha,d).
$$

Given $\delta>0$ and $k\in\mathbb N$, let us consider the set 
\begin{align*}
  G(\alpha,\delta,k)&=\bigcap_{n=k}^\infty P(\alpha,\delta, n)\\
                    &=\bigcap_{n=k}^\infty\left\{x\in X:\left|\frac{1}{n}\sum_{j=0}^{n-1}\varphi(f^{j}(x))-\alpha\right|<\delta \right\}.
\end{align*}
As a consequence of the definition we have that $K_\alpha\subset\bigcup_{k\in\mathbb N} G(\alpha,\delta,k)$.

Now, given $k\in \mathbb{N}$, since $G(\alpha,\delta,k)\subset P(\alpha,\delta,n)$ for $n\geq k$, it follows that $G(\alpha,\delta,k)$ may be covered by $N(\alpha,\delta,n,\varepsilon_j)$ dynamical balls of radius $\varepsilon_j$ and length $n$. Thus, for every $s\geq0$  and $n\geq k$ we have
\[
m(G(\alpha,\delta,k),s,\varepsilon_j)\leq N(\alpha,\delta,n,\varepsilon_j)\exp(-ns).
\]

Let $s=s(\varepsilon_j)>\Lambda_\varphi(\alpha,\varepsilon_j)$ and $\gamma(\varepsilon_j)=(s-\Lambda_\varphi(\alpha,\varepsilon_j))\slash2$. Then, if $\delta_j>0$ is small enough, there exists an increasing sequence
$\{n_\ell\}_{\ell\in\mathbb N}\subset \mathbb N$ such that
\[
N(\alpha,\delta_j,n_\ell, \varepsilon_j)\leq \exp(n_\ell(\Lambda_\varphi(\alpha,\varepsilon_j)+\gamma(\varepsilon_j))).
\] 
Thus, assuming without lost of generality that $n_1\geq k$ and combining the previous observations we conclude that
\[
m(G(\alpha,k,\delta_j),s(\varepsilon
_j),\varepsilon_j)\leq \exp(-n_\ell \gamma(\varepsilon_j)).
\]
In particular, as $\gamma(\varepsilon_j)>0$, letting $n_\ell\to\infty $ we obtain  $m(G(\alpha,k,\delta_j),s(\varepsilon_j),\varepsilon_j)=0$. Consequently, 
\[
h(G(\alpha,k,\delta_j),f,\varepsilon_j)\leq s(\varepsilon_j)
\]
which implies that
\[
h(K_\alpha,f,\varepsilon_j)\leq \sup_k h(G(\alpha,k,\delta_j),f,\varepsilon_j)\leq s(\varepsilon_j).
\]
Hence, 
\begin{align*}
    \mathrm{\overline{mdim}_M}\,(K_\alpha,f,d)&=\limsup_{j\to\infty}\frac{h(K_\alpha,f,\varepsilon_j)}{|\log \varepsilon_j|}\\
    &\leq \limsup_{j\to\infty}\frac{s(\varepsilon_j)}{|\log \varepsilon_j|}\\
    &\leq \limsup_{j\to\infty}\frac{2\gamma(\varepsilon_j)}{|\log \varepsilon_j|}+\limsup_{j\to\infty}\frac{\Lambda_\varphi(\alpha,\varepsilon_j)}{|\log \varepsilon_j|}\\
    &\leq \limsup_{j\to\infty}\frac{2\gamma(\varepsilon_j)}{|\log \varepsilon_j|}+ \Lambda_\varphi\mathrm{\overline{mdim}_M}\,(f,\alpha,d).
\end{align*}
Therefore, as we can choose $s(\varepsilon_j)$ arbitrarily close to $\Lambda_\varphi(\alpha,\varepsilon_j)$, the limsup in the last step is zero for an adequate choice of $s(\varepsilon_j)$.  Then,
$ \mathrm{\overline{mdim}_M}\,(K_\alpha,f,d)\leq \Lambda_\varphi\mathrm{\overline{mdim}_M}\,(f,\alpha,d)$ completing the proof of the proposition.
\end{proof}

\begin{proposition}\label{lemma:21}
Under the hypotheses of Theorem \ref{thm1} we have that
\[
\mathrm {H_\varphi\overline{mdim}_M}\,(f,\alpha,d)\leq \mathrm{\overline{mdim}_M}\,(K_\alpha,f,d).
\]
\end{proposition}

The strategy of the proof consists in constructing a fractal set $F$ contained in $K_\alpha$ and a special probability measure $\eta $ supported on $F$ that satisfies the hypothesis of the so called Entropy Distribution Principle (see Lemma \ref{lemma:entropy-dist-principles}). This will be enough to get the desired inequality. As a step towards the definition of $F$, we introduce a family of finite sets $\mathcal{S}_k$ which play a major role in the construction. 

In order to prove Proposition \ref{lemma:21} we will need the following auxiliary quantity. For $\mu\in\mathcal M_f(X)$, $\delta >0$ and $n\in\mathbb N$, let us denote by $N_\mu(\delta,\varepsilon,n)$ the minimal number $(n,\varepsilon)$-balls needed to cover a set of $\mu$-measure bigger than $1-\delta$. Then, we define
\begin{equation}\label{eq: h mu f eps delta}
h_\mu(f,\varepsilon,\delta)=\limsup_{n\to\infty}\frac{1}{n}\log N_\mu(\delta,\varepsilon,n).
\end{equation}

\begin{proof}[Proof of Proposition \ref{lemma:21}] 
Fix $\gamma>0$ and let $\{\delta_k\}_{k\in\mathbb{N}}$ be a decreasing sequence converging to $0$. Take $\varepsilon
=\varepsilon(\gamma)>0$ and $\mu\in \mathcal M_{f}(X,\varphi,\alpha)$ so that 
\begin{align*}
    \frac{\inf_{|\xi|<5\varepsilon}h_{\mu}(f,\xi)}{|\log5\varepsilon|} \geq\mathrm {H_\varphi\overline{mdim}_M}\,(f,\alpha,d)-\frac\gamma2
\end{align*}
and
\begin{equation}\label{eq:choice varepsilon for h}
\frac{h(K_\alpha, f,\varepsilon/2)}{|\log \varepsilon/2|}\leq \mathrm{\overline{mdim}_M}\,(K_\alpha,f,d) +\gamma.
\end{equation}

Let $\mathcal U$ be a finite open cover of $X$ with diameter $\mathrm{diam}(\mathcal U)\leq 5\varepsilon$ and Lebesgue number $\mathrm{Leb}(\mathcal U)\geq \frac{5\varepsilon}{4}$. We now construct an auxiliary measure which is a finite combination of ergodic measures and ``approximates" $\mu$. To prove this lemma we follow the idea from \cite[p. 535]{You}. In what follows, $\partial \xi$ will denote the boundary of the partition $\xi$ which is just the union of the boundaries of all the elements of the partition and $\xi\succ\mathcal U$ means that $\xi$ refines $\mathcal{U}$, that is, each element of $\xi$ is contained in an element of $\mathcal U$.

\begin{lemma}\label{lem: approx}
For each $k\in\mathbb N$, there exists a measure $\nu_k \in \mathcal M_{f}(X)$ satisfying    
\begin{itemize}
    \item[(a)] $\nu_k=\displaystyle\sum_{i=1}^{j(k)} \lambda_i\nu^k_i$,  where $\lambda_i>0$, $\displaystyle\sum_{i=1}^{j(k)} \lambda_i=1$ and $\nu^k_i\in  \mathcal M_f^{\text{erg}}(X) $;
\item[(b)] $ \displaystyle \inf_{\xi\succ\mathcal U} h_\mu\left(f,\xi\right)\leq \inf_{\xi\succ\mathcal U}\displaystyle h_{\nu_k}\left(f,\xi\right)+\delta_k/2$;
\item[(c)] $\left|\displaystyle\int_X \varphi\; d\nu_k-\displaystyle\int_X\varphi\; d\mu\right|<\delta_k$.
\end{itemize}
\end{lemma}
\begin{proof}[Proof of Lemma \ref{lem: approx}] Given $k\in \mathbb N$, let $\beta_k>0$ be such that for every $\tau_1,\tau_2 \in \mathcal{M}_{f}(X)$,
\begin{displaymath}
d_{\mathcal{M}_{f}(X)}(\tau_1,\tau_2)<\beta_k \implies \left| \int \varphi d\tau_1-\int \varphi d\tau_1\right|<\delta_k
\end{displaymath}
where $d_{\mathcal{M}_{f}(X)}$ is a metric in $\mathcal{M}_{f}(X)$. Let $\mathcal{P}=\{P_1,\ldots,P_{j(k)}\}$ be a partition of $\mathcal{M}_{f}(X)$ whose diameter with respect to $d_{\mathcal{M}_{f}(X)}$ is smaller than $\beta_k$. By the Ergodic Decomposition Theorem there exists a measure $\hat{\mu}$ on $\mathcal{M}_{f}(X)$ satisfying $\hat{\mu}(\mathcal{M}_{f}^{\text{erg}}(X))=1$ such that
$$\int \psi (x) d\mu(x)=\int_{\mathcal{M}_{f}(X)} \left(\int_X \psi(x) d\tau(x)\right)d\hat{\mu}(\tau) \text{ for every } \psi \in C(X,\mathbb R).$$
Let us consider now $\lambda_i=\hat{\mu}(P_i)$ and take $\nu^k_i\in P_i\cap \mathcal{M}_{f}^{\text{erg}}(X)$ such that $\inf_{\xi\succ\mathcal U} h_{\nu^k_i}\left(f,\xi\right)\geq \inf_{\xi\succ\mathcal U} h_\tau\left(f,\xi\right)-\delta_k/2$ for $\hat{\mu}$-almost every $\tau \in P_i\cap \mathcal{M}_{f}^{\text{erg}}(X) $. Observe that such a measure $\nu_i^k$ exists because $\sup_{\tau \in \mathcal{M}_f^{\text{erg}}(X)} \inf_{\xi\succ\mathcal U} h_\tau\left(f,\xi\right)<+\infty$. This latter fact follows from Lemma 3 and Theorem 5 of \cite{Shi} and the fact that the upper metric mean dimension is finite. Finally, define $\nu_k=\sum_{i=1}^{j(k)} \lambda_i\nu^k_i$. It is easy to see that $\nu_k$ satisfies properties a) and c) from the statement. Let us now check that it also satisfies b). By \cite[Proposition 5]{HMRY} we know that
$$\inf_{\xi\succ\mathcal U} h_{\mu}\left(f,\xi\right)=\int_{M_{f}(X)} \inf_{\xi\succ\mathcal U} h_{\tau}\left(f,\xi\right)d\hat{\mu}(\tau).$$
Thus, by our choice of the measures $\nu^k_i$ it follows that
\begin{displaymath}
\begin{split}
    \inf_{\xi\succ\mathcal U} h_{\mu}\left(f,\xi\right)&=\int_{M_{f}(X)} \inf_{\xi\succ\mathcal U} h_{\tau}\left(f,\xi\right)d\hat{\mu}(\tau)\\
    &\leq \sum_{i=1}^{j(k)} \lambda_i \inf_{\xi\succ\mathcal U} h_{\nu_i^k}\left(f,\xi\right)+\delta_k/2\\
    &\leq \inf_{\xi\succ\mathcal U} h_{\nu_k}\left(f,\xi\right) +\delta_k/2
\end{split}
\end{displaymath}
completing the proof of the lemma.
\end{proof}

Let $\nu_k$ be as in the previous lemma. Using the fact that each measure $\nu^k_i$ is ergodic, by the proof of \cite[Theorem 9]{Shi} there exists a finite Borel measurable partition $\xi_k$ which refines $\mathcal U$ so that
\begin{align}\label{eq: relation h part x sep}
   h_{\nu^k_i}(f,5\varepsilon,\gamma)\leq h_{\nu_i^k}(f,\xi_k)\leq h_{\nu^k_i}(f,5\varepsilon/4,\gamma)+\delta_k.
\end{align}
Now, take a finite Borel partition $\xi$ refining $\mathcal{U}$ with $\mu(\partial \xi)=0$ such that
\begin{displaymath}
h_\mu\left(f,\xi\right)-\delta_k\leq \inf_{\zeta \succ\mathcal U}\displaystyle h_{\nu_k}\left(f,\zeta \right).
\end{displaymath}
In particular, since $\xi_k \succ \mathcal{U}$,
\begin{equation}\label{eq: hmu X hnu}
h_\mu\left(f,\xi\right)-\delta_k\leq  h_{\nu_k}\left(f,\xi_k\right).
\end{equation}
Moreover, since $\xi\succ\mathcal U$ it follows that $|\xi|<5\varepsilon$ and thus 
\begin{align}\label{eq:escolha de xi}
    \frac{h_{\mu}(f,\xi)}{|\log5\varepsilon|} \geq\mathrm {H_\varphi\overline{mdim}_M}\,(f,\alpha,d)-\gamma.
\end{align}

Again, since each $\nu^k_i$ is ergodic, there exists $\ell_k\in \mathbb N$ large enough for which the set 
\[
Y_{i}(k)=\left\{x\in X: \left|\frac{1}{n}\sum_{j=0}^{n-1}\varphi(f^j(x))-\int_X\varphi\; d\nu^k_i\right|<\delta_k\;\; \forall \; n\geq \ell_k\right\}
\]
has $\nu^k_i$-measure bigger than $1-\gamma$ for every $k\in\mathbb  N$ and $i\in\{1,\dots,j(k)\}$. 

By \cite[Lemma 3.6]{Tho}, there exists $\hat n_k\to\infty$ with $[\lambda_i\hat n_k]\geq \ell_k$ so that  the maximal cardinality of an $([\lambda_i\hat n_k],5\varepsilon/4)$-separated set in $Y_i(k)$, denoted by $M_{k,i}$, satisfies \begin{equation}\label{eq: Mki inequality}
    M_{k,i}\geq \exp \left( [\lambda_i\hat n_k]\left(h_{\nu^k_i}(f,5\varepsilon/4,\gamma)-\frac{4\gamma}{j(k)}\right)\right) .
\end{equation}
Furthermore,  the sequence $\hat n_k$ can be chosen such that $\hat n_k\geq 2^{ m_k}$ where $m_k=m(\varepsilon/2^{k+2})$ is as in the definition of the specification property. Let $n_k:=m_k({j(k)}-1)+\sum_{i}[\lambda_i\hat n_k]$. Observe that $n_k/\hat n_k\to 1$.

Denote by  $E_{i,k}([\lambda_i\hat n_k],5\varepsilon/4)$ a maximal $([\lambda_i\hat n_k],5\varepsilon/4)$-separated set in $Y_i(k)$. By the specification property, for each
$$x_1\in E_{1,k}(n_1,5\varepsilon/4),\; x_2\in  E_{2,k}(n_2,5\varepsilon/4), \dots,\; x_{j(k)}\in  E_{j(k),k}(n_{j(k)},5\varepsilon/4),$$
there exists $y=y(x_1,\dots,x_{j(k)})\in X$ so that the pieces of orbits
\begin{displaymath}
\{x_i,f(x_i),\dots,f^{[\lambda_i\hat n_k]-1}(x_i); \; i= 1,\ldots,j(k)\}
\end{displaymath}
are $\varepsilon/2^k$-shadowed by $y$ with gap $m_k$. We claim that if $(x_1,\dots,x_{j(k)})\not=(x'_1,\dots,x'_{j(k)})$ then $y(x_1,\dots,x_{j(k)})\not=y'(x'_1,\dots,x'_{j(k)})$. Indeed, if $x_i\not=x'_i$,
\begin{align*}
  \frac{5 \varepsilon}{4} &<d_{[\lambda_i\hat n_k]}(x_i,x'_i)\\
                &\leq d_{[\lambda_i\hat n_k]}(x_i, f^{[\lambda_1\hat n_k]+\dots+[\lambda_{i-1}\hat n_k]+(i -1)m_k}(y))\\
                &+d_{[\lambda_i\hat n_k]}(x'_i, f^{[\lambda_1\hat n_k]+\dots+[\lambda_{i-1}\hat n_k]+(i -1)m_k}(y'))\\
                &+d_{[\lambda_i\hat n_k]}( f^{[\lambda_1\hat n_k]+\dots+[\lambda_{i-1}\hat n_k]+(i -1)m_k}(y), f^{[\lambda_1\hat n_k]+\dots+[\lambda_{i-1}\hat n_k]+(i -1)m_k}(y'))\\
                &<2\frac\varepsilon{2^{k+2}}+d_{[\lambda_i\hat n_k]}( f^{[\lambda_1\hat n_k]+\dots+[\lambda_{i-1}\hat n_k]+(i -1)m_k}(y), f^{[\lambda_1\hat n_k]+\dots+[\lambda_{i-1}\hat n_k]+(i -1)m_k}(y')),
\end{align*}
which implies that $d_{n_k}(y,y')>9\varepsilon/8$ proving our claim. Moreover, as a by-product of this observation we get that 
\begin{align*}
    \mathcal S_k=\{y(x_1,\dots,x_{j(k)}): x_i\in E_{i,k}([\lambda_i\hat n_k],5\varepsilon/4) \text{ for } i=1,\ldots,j(k)\}
\end{align*}
is a $(n_k, 9\varepsilon/8)$-separated set with cardinality $M_k:=\prod_{i=i}^{j(k)}M_{k,i}$. Combining \eqref{eq: relation h part x sep}, \eqref{eq: hmu X hnu} and \eqref{eq: Mki inequality} with the the choices of $\varepsilon$, $\gamma$ and $n_k$ and recalling that $n_k/\hat n_k\to 1$ we get that  for $k$ sufficiently large
\begin{align}\label{eq:200} 
M_k=\prod_{i=i}^{j(k)}\sharp E_{i,k}([\lambda_i\hat n_k],9\varepsilon/8)
    &\geq \exp{\left(\sum_{i=1}^{j(k)}[\lambda_i\hat n_k]\left(h_{\nu^k_i}(f,5\varepsilon/4,\gamma)-\frac{4\gamma}{j(k)}\right)\right)}\\  \nonumber
    &\geq \exp{\left(\hat n_k\sum_{i=1}^{j(k)}\lambda_ih_{\nu^k_i}(f,5\varepsilon/4,\gamma)-4\hat n_k\gamma\right)}\\  \nonumber
    &\geq \exp{\left(\hat n_k\sum_{i=1}^{j(k)}\lambda_ih_{\nu^k_i}(f,\xi_k)-4\hat n_k\gamma-\hat n_k\delta_k\right)}\\  \nonumber
    &\geq \exp{\left(\hat n_k(h_{\nu_k}(f,\xi_k)-4\gamma-\delta_k)\right)}\\  \nonumber
    &\geq \exp{\left(R_k n_k(h_{\nu_k}(f,\xi_k,\gamma)-4\gamma-\delta_k)\right)}\\ \nonumber
    &\geq \exp{\left(R_k n_k(h_{\mu}(f,\xi)-4\gamma - 2\delta_k)\right)}\\ \nonumber
    &\geq \exp{\left(R_k n_k(h_{\mu}(f,\xi)-5\gamma)\right)}
\end{align}
for some $R_k\in(0,1)$.

Let $y=y(x_1,\dots,x_k) \in \mathcal S_k$. Then, 
\begin{align*}
    \left|S_{n_k}\varphi(y)-n_k\alpha\right|
    &\leq \left|S_{n_k}\varphi(y)-n_k\left(\int\varphi d\nu_k-\delta_k\right)\right|\\
&\leq\sum_{i=1}^{j(k)-1}\left|S_{[\lambda_i\hat n_k]}\varphi(f^{\sum_{t=1}^{i-1}[\lambda_t\hat n_k]+(i-1)m_k}(y))-n_k\lambda_i\int\varphi d\nu^k_i\right|\\
&+n_k\delta_k+m_k(j(k)-1)\|\varphi\|\\
&\leq\sum_{i=1}^{j(k)-1}\left|S_{[\lambda_i\hat n_k]}\varphi(x_i)-[\lambda_in_k]\int\varphi d\nu^k_i\right|\\
&+n_k\delta_k+m_kj(k)\|\varphi\|+n_k\mathrm{Var}(\varphi,\varepsilon/2^k)\\
&<\delta_k\sum_{i=1}^{j(k)-1}[\lambda_i\hat n_k]+m_kj(k)\|\varphi\|+n_k\delta_k+n_k\mathrm{Var}(\varphi,\varepsilon/2^k).
\end{align*}
Thus, for sufficiently large $k$,
\begin{align}\label{eq:201}
     \left|\frac{1}{n_k}S_{n_k}\varphi(y)-\alpha\right|\leq \delta_k+\mathrm{Var}(\varphi,\varepsilon/2^k)+\frac{1}{k}.
\end{align}

We now choose a sequence $\{N_k\}_{k\in\mathbb N}$ of positive integers such that $N_1=1$  and
\begin{itemize}
    \item[(1)] $[R_kN_k]\geq 2^{n_{k+1}+m_{k+1}}$, for $k\geq 2$;
    \item[(2)] $[R_{k+1}N_{k+1}]\geq 2^{[R_1N_1n_1]+\dots +[R_kN_k(n_k+m_k)]}$, for $k\geq 1$.
\end{itemize}
Observe that this sequence $\{N_k\}_{k\in\mathbb N}$ grows very fast and 
\begin{align}\label{eq:0}
 \lim_{k\to \infty}\frac{n_{k+1}+m_{k+1}}{R_kN_k}=0 \text{ and }\lim_{k\to \infty}\frac{R_1N_1n_1+\dots +R_kN_k(n_k+m_k)}{R_{k+1}N_{k+1}}=0.
\end{align}
Moreover, we enumerate the points in $\mathcal{S}_k$ as
\begin{displaymath}
\mathcal S_k=\{x_i^k:\; i=1,\ldots,M_k\}.
\end{displaymath}

For any $(i_1,\dots,i_{N_k})\in \{1,2,\dots,M_{k}\}^{[R_kN_k]}$, let $y(i_1,\dots,i_{[R_kN_k]})\in X$ be given by the specification property so that its orbit
$\varepsilon\slash 2^k$-shadows, with gap $m_k$, the pieces of orbits $\{x_{i_j}^k,f(x_{i_j}^k),\dots,f^{n_k-1}(x_{i_j}^k)\}$, $j=1,2,\ldots, [R_kN_k]$. Then, define
$$
\mathcal C_k=\{y(i_1,\dots,i_{[R_kN_k]})\in X:(i_1,\dots,i_{[R_kN_k]})\in \{1,2,\dots,M_k\}^{[R_kN_k]}\}.
$$
Moreover, consider
\begin{equation*}
    c_k=[R_kN_k]n_k+([R_{k}N_k]-1)m_k.
\end{equation*}

We now observe that different sequences in $\{1,2,\dots,M_k\}^{[R_kN_k]}$ give rise to different points in $\mathcal C_k$ and that such points are uniformly separated with respect to $d_{c_k}$.
\begin{lemma}[Lemma 5.1 of \cite{TV}]  \label{lemma:10}
If $(i_1,\dots,i_{[R_kN_k]})\not=(j_1,\dots,j_{[R_kN_k]})$, then
$$
d_{c_k}(y(i_1,\dots,i_{[R_kN_k]}),y(j_1,\dots,j_{[R_kN_k]}))>\varepsilon.
$$
In particular $\sharp\mathcal C_k=M_k^{[R_kN_k]}$.
\end{lemma}

Our next step is to construct inductively an auxiliary sequence of finite sets $\mathcal{T}_k$. 
Let $\mathcal T_1=\mathcal C_1$ and $t_1=c_1$. Now, suppose that we have already constructed the set $\mathcal T_k$ and we will describe how to construct $\mathcal T_{k+1}$. Consider
\begin{align}\label{eq:1}
t_{k+1}&=t_k+m_{k+1}+c_{k+1} \nonumber \\
&\phantom{=} =[R_1N_1]n_1+[R_2N_2](n_2+m_2)+\dots+[R_{k+1}N_{k+1}](n_{k+1}+m_{k+1}).
\end{align}
For $x\in \mathcal T_k$ and $y\in \mathcal C_{k+1}$, let $z=z(x,y)$ be some point such that
\begin{align}\label{eq:4}
    d_{t_k}(x,z)<\frac{\varepsilon}{2^{k+1}} \text{ and } d_{c_{k+1}}(y,f^{t_k+m_{k+1}}(z))<\frac{\varepsilon}{2^{k+1}}.
\end{align}
Observe that the existence of such a point is guaranteed by the specification property of $f$. Then, let us consider
$$
\mathcal T_{k+1}=\{z(x,y):x\in \mathcal T_k, \; y\in \mathcal C_{k+1}\}.
$$
By proceeding as in the proof of the Lemma \ref{lemma:10} we can see that different pairs $(x,y)$, $x\in \mathcal T_k$, $y\in \mathcal C_{k+1}$, produce different points $z=z(x,y)$. In particular, $\sharp \mathcal T_{k+1}=\sharp \mathcal T_k \cdot \sharp\mathcal C_{k+1}$. Therefore, proceeding inductively,
$$
\sharp \mathcal T_k=\sharp \mathcal C_1\dots \sharp\mathcal C_k= M_1^{[R_1N_1]}\dots M_k^{[R_kN_k]}.
$$
In particular, by Lemma \ref{lemma:10} and \eqref{eq:4} we have that for every $x\in  \mathcal T_k$ and $y,y'\in \mathcal C_{k+1}$ with $y\not=y'$,
\begin{align}\label{eq:500}
    d_{t_k}(z(x,y),z(x,y'))<\frac{\varepsilon}{2^{k+2}} \text{ and } d_{t_{k+1}}(z(x,y),z(x,y'))>\frac{3\varepsilon}{4}.
\end{align}

For every $k\in\mathbb N$ let us consider
$$
F_k:=\bigcup_{x\in \mathcal T_k}\overline B_{t_k}(x, \varepsilon/2^{k+1}),
$$
where $\overline B_{ t_k}(x, \varepsilon/2^{k+1})$ denotes the closure of the open ball $ B_{t_k}(x, \varepsilon/2^{k+1})$. As a simple consequence of \eqref{eq:500} we have the following observation.

\begin{lemma}[Lemma 5.2 of \cite{TV}] \label{lemma:11}
For every $k$ the following is satisfied:\\
\noindent(1) for any $x, x'\in \mathcal T_k$, $x \not= x'$, the sets $\overline B_{t_k}(x, \varepsilon/2^{k+1})$ and 
$\overline B_{t_k}(x', \varepsilon/2^{k+1})$ are disjoint;\\
\noindent(2) if $z\in \mathcal T_{k+1}$ is such that $z=z(x,y)$ for some $x \in\mathcal T_k$ and $y\in \mathcal{C}_{k+1}$, then
$$
\overline B_{t_{k+1}}\left(z, \frac{\varepsilon}{2^{k+2}}\right)\subset \overline B_{t_k}\left(x, \frac{\varepsilon}{2^{k+1}}\right).
$$
Hence, $F_{k+1}\subset F_k$.
\end{lemma}

Consider
 $$F:=\bigcap_{k\in\mathbb N}F_k.$$ 
Observe that, since each $F_k$ is a closed and non-empty set and, moreover, $F_{k+1}\subset F_k$, the set $F$ is a non-empty and closed set too. Furthermore, using \eqref{eq:201} we may prove that

\begin{lemma}[Lemma 5.3 of \cite{TV}] \label{lemma: F_subset_K_alpha}
Under the above conditions,
$$F\subset K_\alpha.$$
\end{lemma}

Now, for every $k\geq 1$, let us consider the probability measure $\eta_k$ given by
\begin{align*}
    \eta_k=\frac{1}{\sharp \mathcal T_k}\sum_{z\in \mathcal T_k}\delta_z. 
\end{align*}
Observe that, as $\mathcal{T}_k\subset F_k$,  $\eta_k(F_k)=1$. Moreover,

\begin{lemma}[Lemma 5.4 of \cite{TV}]    \label{lemma:12}
The sequence of probability measures $(\eta_k)_{k\in \mathbb N}$ converges in the weak$^{\ast}$-topology to some probability measure $\eta$. Furthermore, the limiting measure $\eta$ satisfies $\eta(F)=1$.
\end{lemma}

An important feature of the measure $\eta$ that can be obtained by exploring its definition and \eqref{eq:200} is that the $\eta$-measure of some appropriate dynamical balls decay exponentially fast. More precisely, 
\begin{lemma}[Lemma 5.5 of \cite{TV}] \label{lemma: exp decay measure}
For every $n$ sufficiently large and  $q\in X$ so that $B_n\left(q,\frac{\varepsilon}{2}\right)\cap F\not=\emptyset$ one has
$$\eta\left(B_n\left(q,\frac{\varepsilon}{2}\right)\right)\leq \exp{\left(-n(h_\mu(f,\xi)-8\gamma)\right)}.$$
\end{lemma}

In order to conclude our proof we need a simple yet interesting fact whose proof we include for the sake of completeness. This is a version of the \emph{Entropy Distribution Principle} of \cite{TV} (see \cite[Theorem 3.6]{TV}). Observe that for this result, the measure involved does not need to be invariant, as it is the case of the measure $\eta$ obtained in the previous lemmas.
\begin{lemma} \label{lemma:entropy-dist-principles}
Let $f : X \to X$ be a continuous transformation and $\varepsilon >0$. Given a set $Z \subset X$ and a constant $s \geq0$, suppose there exist a constant $C>0$ and a Borel probability measure $\eta$ satisfying:
\begin{itemize}
    \item[(i)] $\eta(Z)>0$;
    \item[(ii)] $\eta(B_n(x,\varepsilon))\leq C e^{-ns}$ for every ball $B_n(x,\varepsilon)$ such that $B_n(x,\varepsilon)\cap Z\not=\emptyset$.
\end{itemize}
Then $h(Z,f,\varepsilon)\geq s$.
\end{lemma}

 \begin{proof}[Proof of Lemma \ref{lemma:entropy-dist-principles}]
Let $\Gamma=\{B_{n_i}(x_i,\varepsilon)\}_i$ be some cover of $Z$. Without loss of generality we may assume that $B_{n_i}(x_i,\varepsilon)\cap Z\not=\emptyset$ for every $i$. In such case we have that
\begin{align*}
    \sum_{i}\exp(-sn_i)&\geq C^{-1}\sum_{i}\eta(B_{n_i}(x,\varepsilon))\geq C^{-1}\eta\left(\bigcup_iB_{n_i}(x,\varepsilon)\right)\\
                       &\geq C^{-1}\eta(Z)>0.
\end{align*}
Therefore, $m(Z,s,\varepsilon)>0$ and hence  $h(Z,f,\varepsilon)\geq s$.
\end{proof}

By Lemma \ref{lemma: F_subset_K_alpha} we have that $h(K_\alpha,f,\varepsilon/2)\geq h(F,f,\varepsilon/2)$.  Lemmas \ref{lemma: exp decay measure} and \ref{lemma:entropy-dist-principles} gives us that $h(F,f,\varepsilon/2)\geq  h_{\mu}(f,\xi)-8\gamma$. Consequently,
 \begin{align*}
  h(K_\alpha,f,\varepsilon/2)\geq 
    h_{\mu}(f,\xi)-8\gamma.\\
 \end{align*}
Thus, combining this observation with \eqref{eq:choice varepsilon for h} and \eqref{eq:escolha de xi} we get that
\begin{align*}
  \mathrm H_\varphi\mathrm{\overline{mdim}_M}\,(f,\alpha,d)-9\gamma
  &\leq  \frac{h_{\mu}(f,\xi)-8\gamma}{|\log5\varepsilon |}\\
  &\leq \frac{h(K_\alpha,f,\varepsilon/2)}{|\log\varepsilon/2|+\log 10}\\
  &\leq \mathrm{\overline{mdim}_M}\,(K_\alpha,f,d) +\gamma.
\end{align*}
Thus, since $\gamma>0$ is arbitrary, the proof of the proposition is complete.
\end{proof}

\begin{proposition}\label{lemma:22}
Under the hypotheses of Theorem \ref{thm1} we have that
$$
\mathrm {H_\varphi\overline{mdim}_M}\,(f,\alpha,d)\geq \Lambda_\varphi\mathrm{\overline{mdim}_M}\,(f,\alpha,d).
$$
\end{proposition}

\begin{proof}
Fix $\gamma>0$. Let $\{\varepsilon_j\}_{j\in\mathbb N}$ be a sequence of positive numbers which converges to zero and satisfies 
$$
\Lambda_\varphi\mathrm{\overline{mdim}_M}\,(f,\alpha,d)=\lim_{j\to\infty}\frac{\Lambda_\varphi(\alpha,\varepsilon_j)}{|\log\varepsilon_j|}.
$$
Then, there exists $\varepsilon_0>0$ so that for all $\varepsilon_j\in(0,\varepsilon_0]$ we have 
$$
\frac{\Lambda_\varphi(\alpha,\varepsilon_j)}{|\log\varepsilon_j|}>\Lambda_\varphi\mathrm{\overline{mdim}_M}\,(f,\alpha,d)-\frac{1}{3}\gamma.
$$
In particular, for every $\varepsilon_j\in(0,\varepsilon_0]$,
$$
\Lambda_\varphi(\alpha,\varepsilon_j)>\left(\Lambda_\varphi\mathrm{\overline{mdim}_M}\,(f,\alpha,d)-\frac{1}{3}\gamma\right)\cdot |\log\varepsilon_j|.
$$

Fix $j\in\mathbb N$ such that $\varepsilon_j\in(0,\varepsilon_0]$. By the alternative expression of $\Lambda_\varphi(\alpha,\varepsilon_j)$ given in \eqref{eq:ineq-sep-cov} it follows that there exists a sequence of positive numbers $(\delta_{j,k})_{k\in \mathbb N}$ converging to zero and such that for every $k\in \mathbb{N}$,
\begin{align*}
 \liminf_{n\to\infty}\frac{1}{n}\log M(\alpha,\delta_{j,k},n,\varepsilon_{j})
 &>\Lambda_\varphi(\alpha,\varepsilon_{j})-\frac{2}{3}\gamma\\
 &>\left(\Lambda_\varphi\mathrm{\overline{mdim}_M}\,(f,\alpha,d)-\frac{1}{3}\gamma\right)\cdot|\log\varepsilon_{j}|-\frac{2}{3}\gamma.
\end{align*}
Similarly, there exists a sequence $(n_{j,k})_{k\in \mathbb{N}}$ in $\mathbb N$ satisfying $\lim_{k\to \infty}n_{j,k}=+\infty $ and 
\begin{align}\label{eq: def Njk}
    M_{j,k}&:=M(\alpha,\delta_{j,k},n_{j,k},\varepsilon_{j})\\
    &\phantom{=} >\exp\left(n_{j,k}\left(\left(\Lambda_\varphi\mathrm{\overline{mdim}_M}\,(f,\alpha,d)-\frac{1}{3}\gamma\right)\cdot|\log\varepsilon_{j}|-\gamma\right)\right). \nonumber
\end{align}

Consider a maximal $(n_{j,k},\varepsilon_j)$-separated set $C_{j,k}$ of $P(\alpha,\delta_{j,k},n_{j,k})$. For each $j,k\in\mathbb N$ consider
$$
\sigma_k^{(j)}=\frac{1}{M_{j,k}}\sum_{x\in C_{j,k}}\delta_x,
$$
and 
$$
\mu_k^{(j)}=\frac{1}{n_{j,k}}\sum_{i=0}^{n_{j,k}-1}(f^i)_\ast(\sigma^{(j)}_k)=\frac{1}{M_{j,k}}\sum_{x\in C_{j,k}}\frac{1}{n_{j,k}}\sum_{i=0}^{n_{j,k}-1}\delta_{f^i(x)}.
$$
It is not difficult to see that any accumulation point of $\{\mu_k^{(j)}\}_{k\in\mathbb N}$, say $\mu^{(j)}$, is $f$-invariant (see \cite[Theorem 6.9]{Walters}). Moreover, $\displaystyle\int_X\varphi \;d\mu^{(j)}=\alpha$ for every $j\in \mathbb{N}$. Indeed,
we may assume without loss of generality that $\displaystyle\lim_{k\to\infty}\mu_k^{(j)}=\mu^{(j)}$. Then, for every $j$ and $k$ in $\mathbb{N}$ we have 
$$
\left|\int_X \varphi\;d\mu^{(j)}_k-\alpha\right|\leq \frac{1}{M_{j,k}}\sum_{x\in C_{j,k}}\left|\frac{1}{n_{j,k}}\sum_{i=0}^{n_{j,k}-1}\varphi (f^i(x))-\alpha\right| \leq \delta_{j,k}. 
$$
Thus,
\begin{align*}
    \left|\int_X \varphi\;d\mu^{(j)}-\alpha\right| &\leq \left|\int_X \varphi\;d\mu^{(j)}-\int_X \varphi\;d\mu^{(j)}_k\right| +\left|\int_X \varphi\;d\mu^{(j)}_k-\alpha\right|\\
                                                   &\leq\left|\int_X \varphi\;d\mu^{(j)}-\int_X \varphi\;d\mu^{(j)}_k\right|+\delta_{j,k}.
\end{align*}
Consequently, making $k\to +\infty$ we conclude that $\displaystyle\int_X\varphi \;d\mu^{(j)}=\alpha$ for every $j\in \mathbb N$ as claimed.

For every $j\in \mathbb N$ choose a Borel partition $\xi(j)=\{A_1,\dots,A_\ell\}$ of $X$ so that $\mathrm{diam}(\xi(j))<\varepsilon_{j}$ and $\mu^{(j)}(\partial A_i)=0$ for $0\leq i\leq \ell$ (see \cite[Lemma 8.5(ii)]{Walters}). Then,
$$H_{\sigma^{(j)}_{k}}\left(\bigvee_{i=0}^{n_{j,k}-1}f^{-i}\xi(j)\right)=\log M(\alpha,\delta_{j,k},n_{j,k},\varepsilon_j).$$
Indeed, observe that if $x$ and $y$ belong to the same element of $\bigvee_{i=0}^{n_k-1}f^{-i}\xi(j)$ then $d_{n_{j,k}}(x,y)<\varepsilon_j$. In particular, no element of $\bigvee_{i=0}^{n_k-1}f^{-i}\xi(j)$ can contain more than one point of a  maximal $(n_{j,k},\varepsilon_{j})$-separated set. Thus, exactly $M(\alpha,\delta_{j,k},n_{j,k},\varepsilon_j)$ elements of $\bigvee_{i=0}^{n_{j,k}-1}f^{-i}\xi(j)$ have $\sigma^{(j)}_k$-measure equal to $\frac{1}{M(\alpha,\delta_{j,k},n_{j,k},\varepsilon_j)}$. All others have zero $\sigma^{(j)}_k$-measure.

Fix natural numbers $q$ and $n_{j,k}$ with $1<q<n_{j,k}$ and define, for $0\leq s\leq q-1$, $a(s)=[(n_{j,k}-s)/q]$ where $[p]$ denotes the integer part of $p$. Fix $0\leq s\leq q-1 $. Then, by \cite[Remark 2(ii), p. 188]{Walters} we have that 
$$
\bigvee_{i=0}^{n_{j,k}-1}f^{-i}\xi(j)=\bigvee_{r=0}^{a(s)-1}f^{-(rq+s)}\left(\bigvee_{i=0}^{q-1}f^{-i}\xi(j)\right)\vee \bigvee_{t\in L}f^{-t}\xi(j)
$$
where $L$ is a set with cardinality at most $2q$. Therefore, using \cite[Theorem 4.3(viii)]{Walters} and \cite[Corollary 4.2.1]{Walters},
\begin{align*}
    \log M(\alpha,\delta_{j,k},n_{j,k},\varepsilon_{j})&=H_{\sigma^{(j)}_{k}}\left(\bigvee_{i=0}^{n_{j,k}-1}f^{-i}\xi(j)\right)\\
 &\leq \sum_{i=0}^{a(s)-1}H_{\sigma^{(j)}_{k}}f^{-(rq+s)}\left(\bigvee_{i=0}^{q-1}f^{-i}\xi(j)\right)+ \sum_{t\in L}H_{\sigma^{(j)}_{k}}(f^{-t}\xi(j)   )\\
 &\leq  \sum_{i=0}^{a(s)-1}H_{\sigma^{(j)}_{k}\circ f^{-(rq+s)}}\left(\bigvee_{i=0}^{q-1}f^{-i}\xi(j)\right)+2q\log(\ell).
\end{align*}
Summing the previous inequality over $s$ from $0$ to $q-1$ and using \cite[Remark 2(iii), p. 188]{Walters}, we get that that
$$
q \log M(\alpha,\delta_{j,k},n_{j,k},\varepsilon_{j})\leq \sum_{p=0}^{n_{j,k}-1}H_{\sigma^{(j)}_{k}\circ f^{-p}}\left(\bigvee_{i=0}^{q-1}f^{-i}\xi(j)\right)+2q^2\log(\ell).
$$
Thus, dividing everything by $n_{j,k}$ in the above inequality and using \eqref{eq: def Njk} and the concavity of the map $\mu\to H_{\mu}(\xi)$ we obtain
\begin{align}\label{eq:02}
  \nonumber &q \left(\left(\Lambda_\varphi\mathrm{\overline{mdim}_M}\,(f,\alpha,d)-\frac{1}{3}\gamma\right)\cdot|\log\varepsilon_{j}|-\gamma\right)\\ \nonumber
    &<\frac{q}{n_{j,k}} \log M(\alpha,\delta_{j,k},n_{j,k},\varepsilon_{j})\\
    &\leq H_{\mu^{(j)}_{k}}\left(\bigvee_{i=0}^{q-1}f^{-i}\xi(j)\right)+\frac{2q^2\log(\ell)}{n_{j,k}}.
\end{align}

Now, since the elements of $\bigvee_{i=0}^{q-1}f^{-i}\xi(j)$ have boundaries of $\mu^{(j)}$-measure zero, it follows from the weak convergence of the measures $\mu^{(j)}_{k}$ to $\mu^{(j)}$ that $\displaystyle\lim_{k\to\infty}\mu^{(j)}_{k}(B)=\mu^{(j)}(B)$
for each element $B$ of $\bigvee_{i=0}^{q-1}f^{-i}\xi(j)$ and, therefore,
$$
\displaystyle\lim_{k\to\infty}H_{\mu^{(j)}_k}\left(\bigvee_{i=0}^{q-1}f^{-i}\xi(j)\right)=H_{\mu^{(j)}}\left(\bigvee_{i=0}^{q-1}f^{-i}\xi(j)\right).
$$
Thus, by \eqref{eq:02} we have that 
\begin{align*}
  \nonumber q \left(\left(\Lambda_\varphi\mathrm{\overline{mdim}_M}\,(f,\alpha,d)-\frac{1}{3}\gamma\right)\cdot|\log\varepsilon_{j}|-\frac{2}{3}\gamma\right)     &\leq H_{\mu^{(j)}}\left(\bigvee_{i=0}^{q-1}f^{-i}\xi(j)\right).
\end{align*}
Dividing both sides of the previous inequality by $q$ and letting $q$ go to $+\infty$ we obtain
$$
\left(\Lambda_\varphi\mathrm{\overline{mdim}_M}\,(f,\alpha,d)-\frac{1}{3}\gamma\right)\cdot|\log\varepsilon_{j}|- \frac{2\gamma}{3}\leq h_{\mu^{(j)}}(f,\xi(j)), \text{ for all }j\in\mathbb N,
$$
which implies that 
$$
\Lambda_\varphi\mathrm{\overline{mdim}_M}\,(f,\alpha,d)-\frac{1}{3}\gamma\leq \frac{h_{\mu^{(j)}}(f,\xi(j))+\frac{2}{3}\gamma}{|\log\varepsilon_{j}|}, \text{ for all }j\in\mathbb N.
$$
Therefore,
\begin{align*}
    \Lambda_\varphi\mathrm{\overline{mdim}_M}\,(f,\alpha,d)-\frac{1}{3}\gamma
    \leq\frac{\inf_{|\xi|<\varepsilon_j} h_{\mu^{(j)}}(f,\xi)+\gamma}{|\log\varepsilon_{j}|} \text{ for all }j\in\mathbb N
\end{align*}
and consequently,
\begin{align*}
    \Lambda_\varphi\mathrm{\overline{mdim}_M}\,(f,\alpha,d)-\frac{1}{3}\gamma
    &\leq\limsup_{j\to\infty}\frac{\sup_{\nu\in\mathcal M_f(X,\alpha,d)}\inf_{|\xi|<\varepsilon_j} h_\nu(f, \xi)}{|\log\varepsilon_{j}|}\\
    &\leq \limsup_{\varepsilon\to0}\frac{\sup_{\nu\in\mathcal M_f(X,\alpha,d)}\inf_{|\xi|<\varepsilon} h_\nu(f,\xi)}{|\log\varepsilon|}\\
    &=\mathrm H_\varphi\mathrm{\overline{mdim}_M}\,(f,\alpha,d)
\end{align*}
completing the proof of the proposition.
\end{proof}

Finally, the first claim of Theorem \ref{thm1} follows directly by combining Propositions \ref{lemma:20}, \ref{lemma:21} and \ref{lemma:22}. For the general case, that is, without the assumption that the quantities $\mathrm{\overline{mdim}_M}\,\Big(K_\alpha,f, d\Big)$, $\Lambda_\varphi\mathrm{\overline{mdim}_M}\,(f,\alpha,d)$ and $\mathrm H_\varphi\mathrm{\overline{mdim}_M}\,(f,\alpha,d)$ are finite, we observe that a simple modification of our proof show us that if one of the quantities is infinite then the other two must also be infinite and, therefore, the first claim of Theorem \ref{thm1} is still true. As for the second claim in Theorem \ref{thm1}, again one can easily see that simple adaptations of the previous proof yield the desired conclusion. The proof of Theorem \ref{thm1} is complete.

\section{Examples}\label{sec: examples}
In this section we present some examples of settings with positive upper/lower metric mean dimension where our results may be applied. Moreover, we also present a simple application of Theorem \ref{thm1} to calculate $\mathrm{mdim_M}\,\Big(K_\alpha,f, d\Big)$.

\begin{example}\label{ex: shift interval}
Let $(Z,D)$ be a compact metric space with upper box-counting dimension $\mathrm{\overline{dim}_B}\;Z<\infty$. Let us consider $X=Z^\mathbb N$ endowed with the metric 
$$
d((x_n)_{n\in\mathbb N},(y_n)_{n\in\mathbb N})=\sum_{n=1}^\infty\frac{1}{2^n}D(x_n,y_n)
$$
and let $\sigma: X\to X$ be the shift map. It is well known that $\sigma$ has the specification property and 
 $\mathrm{\overline{mdim}_M}\,(X,\sigma,d)=\mathrm{\overline{dim}_B}\;Z$ and  $\mathrm{\underline{mdim}_M}\,(X,\sigma,d)=\mathrm{\underline{dim}_B}\;Z$ (see for instance \cite{Acevedo-Rodrigues}). In particular, we may apply Theorem \ref{thm1} to it getting, for instance, that for any $\varphi\in C^0(X,\mathbb R)$ and $\alpha\in\mathbb R$,
$$
\mathrm{\overline{mdim}_M}\,\Big(K_\alpha,\sigma, d\Big)=\Lambda_\varphi\mathrm{\overline{mdim}_M}\,(\sigma,\alpha,d)= \mathrm{H_\varphi\overline{mdim}_M}\,(\sigma,\alpha,d).
$$
\end{example}\label{ex1}

\begin{example}
Let $X=[0,1]^{\mathbb N}$ be endowed with the metric induced by the Euclidian distance in $[0,1]$ as in the previous example and consider the set 
$$
E=\left\{\{x^{(i,j)}\}_{i,j\in\mathbb N}\in X:x^{(i,j)}_n=\frac{1}{2^j} \text{ if }i=n\text{ and }x^{(i,j)}_n=0 \text{ if }i\not=n\right\}\cup\{e\},
$$
where $e=(0,0,\dots)$, which is closed and shift invariant. If $2^E$ denotes the space of subsets of $X$ endowed with the Hausdorff distance $d_H$, by \cite[Proposition 3.6]{HW} we have that 
$$
\mathrm{\overline{mdim}_{M}}\,\Big(E,\sigma,d\Big)=0 \text{ and } \mathrm{mdim_{M}}\,\Big(2^E,\sigma_\sharp,d_H\Big)=1,
$$
where $\sigma_\sharp$ is the induced map by $\sigma$ on the hyperspace $2^E$. By \cite[Proposition 4]{BAS} we have that $\sigma_\sharp$ has the specification property and then Theorem \ref{thm1}  may be applied.
\end{example}

\begin{example}
 It was proved in \cite{ARA,CRV} that for $C^0$-generic homeomorphisms acting on a compact and smooth manifold $X$ with dimension greater than one, the upper metric mean dimension with respect to the smooth metric coincides with the dimension of the manifold. Moreover, they also proved that the set of homeomorphisms with positive lower metric mean dimension is $C^0$ dense in the set of homeomorphisms of $X$. Now, in order to be able to apply Theorem \ref{thm1} to elements of those sets, we need to guarantee that they have the specification property. For this purpose we restrict ourselves to the set of conservative homeomorphisms, where the specification property holds $C^0$-generically. 
 
 We fix a good Borel probability measure $\mu\in\mathcal M(X)$,  i.e., a probability measure that satisfies the following conditions:
\medskip

\noindent $(C_1)\quad$ [Non-atomic] For every $x \in X$ one has $\mu(\{x\})=0$;
\smallskip

\noindent $(C_2) \quad$ [Full support] For every nonempty open set $U \subset X$ one has $\mu(U)>0$;
\smallskip

\noindent $(C_3) \quad$ [Boundary with zero measure] $\mu(\partial X)=0$.
\medskip\\
In a forthcoming paper by S. Roma\~na and G. Lacerda it is proved that there exists a Baire generic  subset of $\mathrm{Homeo}_\mu(X,d)$ (the set of conservative homeomorphisms on $X$) with metric mean dimension equal to the dimension of $X$. Consequently, since according to \cite{GL} the specification property is a Baire generic property in $\mathrm{Homeo}_\mu(X,d)$, there exists a $C^0$-open and dense subset of $\mathrm{Homeo}_\mu(X,d)$ whose elements have positive upper metric mean dimension and the specification property and, in particular, Theorem \ref{thm1} may be applied to those elements.
\end{example}

In the next two examples we consider the specification property for linear operators acting on Banach spaces and we start by recalling the appropriate definition for this setting. Let $B$ be a Banach space over $\mathbb K$ ($= \mathbb R$ or $\mathbb{C}$) and $T:B\to B$ be a linear operator. We say that $T$ has the \emph{operator specification property} if there exists a sequence of $T$-invariant sets $\{K_m\}_{m\in\mathbb N}$ with $B=\overline{\cup_{m\in \mathbb N}K_m}$ for which $T|_{K_m}:K_m\to K_m$ satisfies the specification property. We emphasize that the sets $K_m$ do not need be compact, although in the all known examples we have compactness for such sets.
\begin{example}
Fix $\nu=(\nu_n)_{n\in\mathbb N}\in \mathbb R^{\mathbb N}$ so that $\nu_n>0$ for all $n\in\mathbb N$ and  $\displaystyle\sum_{n=1}^\infty\nu_n<\infty$. Given $1\leq p<+\infty$, consider 
$$
\ell^p(\nu)=\left\{(x_n)_{n\in\mathbb N}\in \mathbb K^{\mathbb N}: \|(x_n)_{n\in\mathbb N}\|_{\ell_p(\nu)}:=\left(\sum_{n=1}^\infty|x_n|^p\nu_n\right)^{\frac{1}{p}}<\infty\right\},
$$
which is a Banach space, and the shift map $\sigma:\ell^p(\nu)\to \ell^p(\nu)$. By \cite[Theorem 2.1]{BMP} we have that $\sigma:\ell^p(\nu)\to\ell^p(\nu)$
has the operator specification property with $K_m=mK$, $m\in \mathbb{N}$, where $K$ is the compact set $K=\{(x_n)_{n\in\mathbb N}\in \ell^p(\nu): |x_n|\leq 1 \text{ for all }n\in\mathbb N\}$. We now observe that $T|_{K}:K\to K$ has positive metric mean dimension. More precisely,

\begin{lemma}\label{lem: metric mean Linear op}
$$
\mathrm{\overline{mdim}_M}\,\Big(K,\sigma, \|\cdot\|_{\ell_p(\nu)}\Big)=\mathrm{\underline{mdim}_M}\,\Big(K,\sigma, \|\cdot\|_{\ell_p(\nu)}\Big)=1.
$$ 
\end{lemma}
\begin{proof}
Given $\varepsilon >0$ and $n\in\mathbb N$, we observe that 
$$
\left\{x\in K: x_i\in\left \{0,\frac{\varepsilon}{\sqrt[p]{\nu_1}},\frac{2\varepsilon}{\sqrt[p]{\nu_1}},\dots,\left\lfloor1/\frac{\varepsilon}{\sqrt[p]{\nu_1}}\right\rfloor\frac{\varepsilon}{\sqrt[p]{\nu_1}}\right\} \text{ for all }1\leq i \leq n\right\}
$$
is a $(n,\varepsilon)$-separated set in $K$. In particular,
\begin{equation}\label{eq: aux example4.1}
\begin{split}
\mathrm{\underline{mdim}_M}\,\Big(K,\sigma, \|\cdot\|_{\ell_p(\nu)}\Big)&=\liminf_{\varepsilon\to0}\frac{h(\sigma,\varepsilon)}{|\log\varepsilon|}\\
&\geq \liminf_{\varepsilon\to0}\frac{\limsup_{n\to\infty}\frac1n\left|\log \left(\left\lfloor1/\frac{\varepsilon}{\sqrt[p]{\nu_1}}\right\rfloor\right)^n\right|}{|\log\varepsilon|}
=1. 
\end{split}
\end{equation}

In order to get the reverse inequality, let $\ell\in\mathbb N$ be so that $\displaystyle\sum_{n\geq \ell}\nu_n<\frac{\varepsilon}{2}$ and define $M=\left(\displaystyle\sum_{k\in \mathbb N}\nu_k\right)^{1/p}>0$. We consider an open cover of $[-1,1]$ by 
$$
I_k=\left(\frac{(k-1)\varepsilon}{12M},\frac{(k+1)\varepsilon}{12M}\right), \text{ for }-\lfloor12M/\varepsilon\rfloor\leq k \leq \lfloor12M/\varepsilon\rfloor.
$$
Note that each $I_k$ has length $\displaystyle\frac{\varepsilon}{6M}$. Given $n\geq 1$, let us consider the following 
open cover of $K^\mathbb{N}$:
$$
\{x:x_1\in I_{k_1},x_2\in I_{k_2},\dots,x_{n+\ell}\in I_{k_{n+\ell}}\},
$$
where  $-\lfloor12M/\varepsilon\rfloor\leq k_1,\dots,k_{n+\ell} \leq \lfloor12M/\varepsilon\rfloor$.
Observe that each element of this open cover has diameter less than $\varepsilon$ with respect to the metric $d_n$ (induced by $\|\cdot\|_{\ell_p(\nu)}$). 
So,
\begin{align*}
  \limsup_{\varepsilon\to0}\frac{h(\sigma,\varepsilon)}{|\log\varepsilon|}
  \leq \limsup_{\varepsilon\to0}\frac{\limsup_{n\to\infty}\frac1n\log \left(2\left\lfloor12M/\varepsilon\right\rfloor\right)^{n+\ell+1}}{|\log\varepsilon|}=1.  
\end{align*}
Hence,
\begin{equation}\label{eq: auxil example4.2}
    \mathrm{\overline{mdim}_M}\,\Big(K,\sigma, \|\cdot\|_{\ell_p(\nu)}\Big)\leq 1.
\end{equation}
Finally, combining \eqref{eq: aux example4.1} and \eqref{eq: auxil example4.2} we get the desired result.
\end{proof}

As a consequence of the previous proof we also get that
$$
\mathrm{mdim_M}\,\Big(K_m,\sigma, \|\cdot\|_{\ell_p(\nu)}\Big)=1
$$ 
for  all $m\in\mathbb N$. In particular, we may apply Theorem \ref{thm1} to $\sigma|K_m:K_m\to K_m$ for every $m\in \mathbb N$.

\end{example}

\begin{example}\label{ex: weighted shifts}
Another class of examples is given by the weighted shifts. Let $\nu=(\nu_n)_{n\in\mathbb N}=(1)_{n\in\mathbb N}$ and consider $\ell^p(\nu)$ as in the previous example. Observe that in this case $\ell^p(\nu)=\ell^p$. Let $w=(w_n)_{n\in\mathbb N}$ be a weight sequence and define the \emph{weighted shift} on $\ell^p$ as 
$B_w((x_n)_{n\in\mathbb N})=(w_{n+1}x_{n+1})_{n\in\mathbb N}$. It was observed in \cite[p. 602]{BMP} that if one considers $a=(a_n)_{n\in\mathbb N}$ given by 
$$
a_1=1 \text{ and }a_n:=w_2\dots w_n, \text{ for all }n>1,
$$
and $\bar \nu = (\bar \nu_n)_{n\in \mathbb N}$ given by
$$
\bar\nu_n=\frac{1}{\prod_{j=2}^n|w_j|^p}, \text{ for all }n\in\mathbb N, 
$$
then 
$$
\phi_a:(x_{n})_{n\in\mathbb N}\in \ell^p\mapsto \phi_a((x_{n})_{n\in\mathbb N})=(a_1x_1,a_2x_2,\dots)\in \ell^p(\bar\nu)
$$
defines a topological conjugacy between the weighted shift and the backward shift given in the previous example. Moreover, they observed that this topological conjungacy is also an isometry, which implies that 
\[\mathrm{\overline{mdim}_M}\,\Big(\phi_a^{-1}(K_m),B_w, \|\cdot\|_{\ell_p}\Big) =\mathrm{\underline{mdim}_M}\,\Big(\phi_a^{-1}(K_m),B_w, \|\cdot\|_{\ell_p}\Big)=1\]
for all $m\in\mathbb N$. Furthermore, if $\sum_{n=1}^\infty\bar\nu_n<\infty$ we have that $B_w$ has the operator specification property (see \cite[Theorem 2.3]{BMP}) and then we are in the context of  Theorem \ref{thm1}.
\end{example}

\begin{example}
    Let us consider $X=[0,1]^{\mathbb Z}$ endowed with the metric 
  \[
d((x_n)_{n\in\mathbb Z},(y_n)_{n\in\mathbb Z})=\sum_{n\in \mathbb Z}\frac{1}{2^{|n|}}|x_n-y_n|
\]
and let $\sigma: X\to X$ be the left shift map. Similarly to Example \ref{ex: shift interval}, $\sigma$ has the specification property and, moreover, $\mathrm{mdim_M}\,\left(X,\sigma, d\right)=1$. In particular, Theorem \ref{thm1} may be applied in this setting. Let $\lambda$ be the Lebesgue measure on $[0,1]$ and consider $\mu=\lambda^{\mathbb Z}$. Then, it is a well known fact that $\mu$ is ergodic. Given $\varphi\in C^0(X,\mathbb R)$, take $\alpha=\int \varphi d\mu$. We will show that $\mathrm{mdim_M}\,\left(K_\alpha,\sigma, d\right)=1$. In order to do it we recall the definition of \emph{Brin-Katok local entropy}: for an ergodic measure $\mu \in \mathcal{M}_\sigma(X)$, $\varepsilon>0$ and a point $x \in X$, let us consider
\[
h_{\mu}^{BK}(\varepsilon,x)=\limsup_{n \to \infty}-\frac{1}{n}\log\mu(B_n(x,\varepsilon)),
\]
where $B_{n}(x,\varepsilon)$ is defined as in Section \ref{sec: def and state}. Since $\mu$ is ergodic, the map $x \mapsto h_{\mu}^{BK}(\varepsilon,x)$ is constant $\mu$-almost everywhere. Denote this constant by $h_{\mu}^{BK}(\varepsilon)$. Then, we have the following observation.
\begin{lemma}[See \cite{GS}.]\label{lem: Brin Katok lower bound}
For any ergodic measures $\mu \in \mathcal{M}_\sigma(X)$ and any $\varepsilon>0$,
\[
h_{\mu}^{BK}(\varepsilon)\leq\inf_{|\xi|<\varepsilon}h_{\mu}(\sigma,\xi),
\]
where the infimum is taken over all finite measurable partitions of $X$ with diameter smaller than $\varepsilon$.
\end{lemma}
Therefore, considering the measure $\mu=\lambda^{\mathbb Z}$ given above and using Theorem \ref{thm1} we get that
\[
\begin{split}
\mathrm{\underline{mdim}_M}\,\Big(K_\alpha,\sigma, d\Big)&= \mathrm H_\varphi\mathrm{\underline{mdim}_M}\,(f,\alpha,d)\\
&\geq \liminf_{\varepsilon\to0}\frac{1}{|\log\varepsilon|}\inf_{|\xi|<\varepsilon}h_\mu(f,\xi)\\
&\geq \liminf_{\varepsilon\to0}\frac{1}{|\log\varepsilon|}h_{\mu}^{BK}(\varepsilon).\\
\end{split}
\]
Now, in \cite[Example 12]{Shi} it is proved that $\liminf_{\varepsilon\to0}\frac{1}{|\log\varepsilon|}h_{\mu}^{BK}(\varepsilon)\geq 1$. Consequently,
\[
\begin{split}
1&=\mathrm{mdim_M}\,\left(X,\sigma, d\right)= \mathrm{\overline{mdim}_M}\,\left(X,\sigma, d\right)\\
&\geq \mathrm{\overline{mdim}_M}\,\left(K_\alpha,\sigma, d\right)\geq \mathrm{\underline{mdim}_M}\,\Big(K_\alpha,\sigma, d\Big) \geq 1
\end{split}
\]
and thus, $\mathrm{mdim_M}\,\left(K_\alpha,\sigma, d\right)= \mathrm{\overline{mdim}_M}\,\left(K_\alpha,\sigma, d\right)= \mathrm{\underline{mdim}_M}\,\left(K_\alpha,\sigma, d\right) =1$ as claimed. We observe that our Theorem \ref{thm1} combined with Lemma \ref{lem: Brin Katok lower bound} may be very useful for giving lower bounds for $\mathrm{\overline{mdim}_M}\,\Big(K_\alpha,\sigma, d\Big)$ and $\mathrm{\underline{mdim}_M}\,\Big(K_\alpha,\sigma, d\Big)$. In fact, it is enough to take an ergodic measure $\mu$ satisfying $\alpha=\int \varphi d\mu$ and estimate $h_{\mu}^{BK}(\varepsilon)$ which, in general, may be easier than estimating $\inf_{|\xi|<\varepsilon}h_{\mu}(\sigma,\xi)$.
\end{example}

\medskip{\bf Acknowledgements.} 
We would like to thank the referees for their useful comments that helped us to improve our paper. L.B. was partially supported by a CNPq-Brazil PQ fellowship under Grant No. 307633/2021-7.


\end{document}